\documentclass[12pt, reqno, twoside, letterpaper]{amsart}

\usepackage{paperstyle}
\newcommand{\be}{\begin{equation}}
\newcommand{\ee}{\end{equation}}

\title[Consecutive primes and IP sets]
      {Consecutive primes and IP sets}
      
\author[W.\ D.\ Banks]{William D.\ Banks}

\address{Department of Mathematics, 
         University of Missouri, 
         Columbia MO, USA.}

\email{bankswd@missouri.edu}
        
\date{\today}

\begin{document}

\thanks{
MSC Primary: 11N05; Secondary: 11N36, 11B75.}


\maketitle

\begin{abstract}
For an infinite set $\cM\subset\N$, let ${\tt FS}(\cM)$
be the set of all nonzero finite sums of distinct numbers in $\cM$.
An \emph{IP set} is any set of the form ${\tt FS}(\cM)$.
Let $p_n$ denote the $n$-th prime number for each $n\ge 1$.
A \emph{de Polignac number} is any number $m$ such that $p_{n+1}-p_n=m$
for infinitely many $n$.
In this note, we show that every IP set  $\cA\subset 2\N$ contains infinitely many
de~Polignac numbers.
\end{abstract}


\def\R{\mathbb R}
\def\Z{\mathbb Z}
\def\N{\mathbb N}
\newcommand{\nearint}[1]{\left\llbracket #1 \right\rrbracket}

\section{Introduction}

A classical result of Hindman~\cite{Hindman}
(answering a question posed by Graham and 
Rothschild~\cite{GrahamRoths}) relates to \emph{finite sums sets},
which are those of the form
\[
{\tt FS}(\cM)\defeq\Big\{\sum_{m\in\cM'}m:
\varnothing\ne\cM'\subset\cM,~|\cM'|<\infty\Big\}
\]
for some infinite set $\cM$ of natural numbers. Hindman showed that
for any finite coloring of $\N$, there is an
infinite set $\cM$ such that ${\tt FS}(\cM)$ is monochromatic.
In their treatise on topological dynamics and combinatorial number theory,
Furstenberg and Weiss~\cite{FurstWeiss} coined the term \emph{IP set} to
mean any finite sums set ${\tt FS}(\cM)$.
There are numerous interesting examples of IP sets. 
Bergelson and Rusza~\cite{BergRusza} gave a general method for constructing
a variety of IP sets; they showed that if $\cS$ is the set
of squarefree natural numbers and $s\in\N$, then
$\cS-s$ contains an IP~set if and only if $s\in\cS$.

 Let $p_1\defeq 2$, $p_2\defeq 3$, $p_3\defeq 5$, etc.,
be the sequence of all primes. Our aim in this note is to establish
a connection between IP sets and the sequence $(p_{n+1}-p_n)_{n\in\N}$
of consecutive prime gaps.

\begin{theorem}\label{thm:simple}
For every IP set $\cA\subset 2\N$, there exist infinitely many $a\in\cA$
with the property that $p_{n+1}-p_n=a$ for infinitely many $n$.
\end{theorem}

\begin{corollary}\label{cor:one}
For every $s\in\cS$, there are infinitely many
$m\in\cS-s$ such that $p_{n+1}-p_n=m$ for infinitely many $n$.
\end{corollary}

\begin{corollary}\label{cor:two}
For every $c>0$, there are infinitely many numbers $m$ in the set
\[
\{m\in\N:p\mid m\Longrightarrow p>c\}
\]
with the property that
$\tfrac12(p_{n+1}-p_n)=m-1$ for infinitely many $n$.
\end{corollary}

\begin{corollary}\label{cor:three}
There are infinitely many $m\in\N$ that have
only the digits~$0$ and~$2$ in their decimal expansion
and satisfy $p_{n+1}-p_n=m$ for infinitely many~$n$.
\end{corollary}

Theorem~\ref{thm:simple} is a simple consequence of our main result
(Theorem~\ref{thm:hard}) whose formulation uses 
terminology given in the next section; the corollaries follow at once
by choosing appropriate IP sets. The proof of Theorem~\ref{thm:hard}
is straightforward and short, but it relies on several deep results
from number theory and additive combinatorics. Thus, the ``heavy lifting''
has largely been accomplished thanks to the efforts of various mathematicians;
their contributions are outlined in the next section.

\section{Background Material}
\label{sec:background}

We briefly review background material associated with our main theorem and
some results needed in the proof.

\subsection{de~Polignac numbers}\label{sec:deP}
The notorious \emph{twin prime conjecture} asserts the existence
of infinitely many natural numbers $n$ for which $p_{n+1}-p_n=2$.
This conjecture first appeared in an 1849 monograph of
Alphonse de~Polignac \cite{deP}, who also speculated that
for every even natural number $2m$, there are infinitely many prime
gaps of size $2m$, i.e., that $p_{n+1}-p_n=2m$ holds for infinitely many $n$.
Let {\bf Pol} denote the set of even integers $2m$ with this property.
The elements of {\bf Pol} are called \emph{de~Polignac numbers},
and de~Polignac's hypothesis that ${\bf Pol}=2\N$ is known
as \emph{Polignac's conjecture}.\footnote{Another equally famous
conjecture of Alphonse de Polignac is that every odd integer $k\ge 3$
can be expressed as the sum of a power of two and a prime number. Thus,
in the literature, the term \emph{de Polignac number} often refers to
odd integers $k$ of this form.}

The \emph{bounded gap conjecture} (that is, the assertion
that ${\bf Pol}\ne\varnothing$) was first proved in the
celebrated 2013 paper of Zhang \cite{Zhang}, who employed
a modification of the Goldston-Pintz-Y{\i}ld{\i}r{\i}m sieve \cite{GPY}
to prove the existence of a de~Polignac number not exceeding
$7\times 10^7$. The {\tt Polymath8} project sought to understand
and improve Zhang's arguments 
and ultimately showed that ${\bf Pol}\cap[2,4680]\ne\varnothing$.
Later in 2013, Maynard~\cite{Maynard}
showed that for every $m\in\N$ one has\footnote{The same
result, albeit with a slightly weaker bound, had been independently proven
by Tao around the same time; the latter work remains unpublished.}
\[
\liminf_n(p_{n+m}-p_m)\ll m^3\er^{4m}.
\]
His breakthrough paper also proved the existence
of infinitely many prime gaps of size at most $600$; in other words,
it showed that ${\bf Pol}\cap[2,600]\ne\varnothing$. 
Following Maynard's announcement, a new polymath project was launched
({\tt Polymath8b}) and eventually achieved the result that
${\bf Pol}\cap[2,246]\ne\varnothing$, currently the strongest known
result in this direction.

\subsection{The Maynard-Tao theorem}

An ordered tuple $\cH$ of distinct nonnegative integers is said to be
\emph{admissible} if it avoids at least one residue class mod $p$ for 
every prime $p$. Following Tao and Ziegler~\cite{TZ}, we say that
a finite admissible tuple $\cH=(h_1,\ldots,h_k)$ is \emph{prime-producing}
if there are infinitely many $n\in\N$ such that
$\{n+h_1,\ldots,n+h_k\}$ are simultaneously prime. The
\emph{Dickson-Hardy-Littlewood conjecture} asserts that every
such tuple $\cH$ is prime-producing. This conjecture remains one of
the great unsolved problems in number theory, and the strongest
unconditional result in this direction is the following theorem of
Maynard~\cite{Maynard} and~Tao.

\begin{theorem}[Maynard--Tao]\label{thm:MT}
For every integer $m \ge 2$, there is a number $k_m$ for which
the following holds. If $(h_1,\ldots,h_k)$ is admissible with $k\ge k_m$,
then the set $\{n+h_1,\ldots,n+h_k\}$ contains at least $m$ primes
for infinitely many $n\in\N$.
\end{theorem}

Shortly after this theorem was announced, a version of the Maynard-Tao theorem
producing \emph{consecutive} primes in tuples was given by 
Banks, Freiberg and Turnage-Butterbaugh \cite{BFTB}, who used the result to
solve an old problem of Erd\H os. The following is a slightly
weakened version of their theorem.

\begin{theorem}[Banks--Freiberg--Turnage-Butterbaugh]\label{thm:BFTB}
Fix an integer $m\ge 2$, and let $k_m$ have the property stated in
Theorem~\ref{thm:MT}. If $(h_1,\ldots,h_k)$ is admissible with
$k\ge k_m$, then there is a set
$\{h'_1,\ldots,h'_m\}\subset\{h_1,\ldots,h_k\}$ such that the set
$\{n+h'_1,\ldots,n+h'_m\}$ consists of $m$ consecutive primes
for infinitely many $n\in\N$.
\end{theorem}

%

\subsection{Ramsey's theorem}
We recall the well known result of Ramsey~\cite{Ramsey}.

\begin{theorem}[Ramsey]\label{thm:ramsey}
For any $r,s\in\N$, there is a least natural number $R(r,s)$
such that every blue--red edge coloring of the complete graph
on $R(r,s)$ vertices contains either a blue clique on $r$ vertices or a
red clique on $s$ vertices.
\end{theorem}

\section{Statement of the main theorem}
\label{sec:2nd theorem}

The present paper was inspired by recent work
of Tao and Ziegler~\cite{TZ}, who study infinite restricted sumsets
in the shifted primes and ascending chains of prime-producing tuples.
Although our results revolve around the study of sums of de~Polignac numbers,
we do not use any results from \cite{TZ} but instead rely only on
the theorems discussed in \S\ref{sec:background}
and an auxiliary result (see Lemma~\ref{lem:admissible} below).

\begin{theorem}\label{thm:hard}
Let $\cA\subset 2\N$ be an IP set.
Then, for any $k\ge 1$, there is a sequence $a_1,a_2,\ldots,a_k$ in $\cA$
such that
\be\label{eq:yo}
\sum_{i\le n\le j}a_n\in{\bf Pol}\qquad\text{whenever}\quad
1\le i\le j\le k.
\ee
Moreover, the $a_j$'s can be chosen so that all of
the sums in \eqref{eq:yo} are distinct from one another.
\end{theorem}

Note that Theorem~\ref{thm:hard} implies that
$|\cA\cap{\bf Pol}|=\infty$, and Theorem~\ref{thm:simple} follows.

\section{Proof of the main theorem}
\label{sec:proof}

We need the following lemma.

\begin{lemma}\label{lem:admissible}
Every infinite set $\cB\subset\N$
contains a strictly increasing sequence $(b_n)_{n\in\N}$ 
for which the tuple $(b_1,b_2,\ldots)$ is admissible.
\end{lemma}

\begin{proof}
Put $b_0\defeq 0$ and $\cB_0\defeq\cB$. Using induction, we construct
a sequence $(b_n)_{n\in\N}$ contained in $\cB$ and a family
of sets $\{\cB_n\}_{n\in\N}$ satisfying
\be\label{eq:Scondn}
b_n>b_{n-1},\qquad \cB_n=\cB_{n-1}\cap(p_n\Z+b_n)
\mand
|\cB_n|=\infty.
\ee
Indeed, let $n\ge 1$, and suppose $b_j$ and $\cB_j$ have been determined
for all $j<n$. Since $\cB_{n-1}$ is infinite, the pigeonhole principle
shows that $\cB_n\defeq\cB_{n-1}\cap (p_n\Z+h)$ is infinite
for at least one integer $h\in[0,p_n)$. Replacing $h$ by any number
$b_n\in\cB_n$ larger than $b_{n-1}$, we obtain \eqref{eq:Scondn},
completing the induction.

To finish the proof, we show that $(b_1,b_2,\ldots)$ is admissible.
Indeed, let $p_n$ be an arbitrary prime, and observe that
$\cB_m\subset\cB_n\subset p_n\Z+b_n$
for all integers $m\ge n$; consequently,
\[
\big|\{b_1,b_2,\ldots\}\bmod p_n\big|
=\big|\{b_1,\ldots,b_n\}\bmod p_n\big|\le n<p_n,
\]
and thus $(b_1,b_2,\ldots)$ avoids at least
one residue class mod $p_n$.
\end{proof}

\begin{proof}[Proof of Theorem~\ref{thm:hard}]
As $\cA$ is an IP set,  $\cA={\tt FS}(\cM)$ for some infinite set $\cM\subset\N$.
We order the elements $\{m_1,m_2,\ldots\}$ of $\cM$ with
$m_1<m_2<\cdots$ and denote by $\cB$ the set of all partial sums
of the form $\sum_{k\le \ell}m_k$.
According to Lemma~\ref{lem:admissible}, the set $\cB$
contains a strictly increasing sequence $(b_n)_{n\in\N}$ 
for which the tuple $(b_1,b_2,\ldots)$ is admissible. Note that,
for any pair $(i,j)$ with $1\le i<j$, there are numbers
$\ell'_i<\ell'_j$ such that
\[
b_j-b_i=\sum_{\ell'_i<k\le\ell'_j}m_k\in\cA.
\]

Next, we extract from $(b_n)_{n\in\N}$
a lacunary subsequence $(v_n)_{n\in\N}$
such that $v_{n+1}/v_n>10$ (say) for each $n$; this ensures that the
differences $v_j-v_i$ are all distinct from one another.
The resulting tuple $(v_1,v_2,\ldots)$ is also admissible.
As before, for any pair $(i,j)$ with $1\le i<j$, there are numbers
$\ell_i<\ell_j$ such that
\be\label{eq:v sight}
v_j-v_i=\sum_{\ell_i<k\le\ell_j}m_k\in\cA.
\ee

Now, in the notation of Theorems~\ref{thm:BFTB} and~\ref{thm:ramsey}, we set
\[
r\defeq L+1,\qquad
s\defeq k_2,
\mand
N\defeq R(r,s).
\]
We now view the finite collection $\{v_1,\ldots,v_N\}$ as the vertex set
$\cV$ of a complete graph~$\cG$ with edge set is
$\cE\defeq\{e_{i,j}:1\le i<j\le N\}$;
each edge $e_{i,j}$ connects vertex $v_i$ to vertex $v_j$.
Define a blue--red coloring of the edges via the map
\[
\sC:\cE\to\{{\tt blue},{\tt red}\},\qquad
\sC(e_{i,j})\defeq\begin{cases}
{\tt blue}&\quad\hbox{if $v_j-v_i\in{\tt Pol}$},\\
{\tt red}&\quad\hbox{if $v_j-v_i\not\in{\tt Pol}$}.\\
\end{cases}
\]
By Theorem~\ref{thm:ramsey}, the graph $\cG$ contains either a blue clique on
$r=L+1$ vertices or a red clique on $s=k_2$ vertices.

We claim that the latter possibility is untenable, i.e.,
no such red clique exists.
Indeed, suppose on the contrary that there is a
complete subgraph $\cG^\flat\subset\cG$ determined by vertices
$\cV^\flat\defeq\{h_1,\ldots,h_{k_2}\}$ in $\cV$ such that $\cG^\flat$
is monochromatic and colored red by the map $\sC$.
By Theorem~\ref{thm:BFTB}, there is a set
$\{h'_1,h'_2\}\subset \cV^\flat$
such that $n+h'_1$ and $n+h'_2$ are consecutive primes for infinitely
many $n\in\N$; in other words, $|h'_2-h'_1|\in{\bf Pol}$.
But this is not possible
since $h'_1$ and $h'_2$ are distinct vertices in $\cV^\flat$
and hence the corresponding edge is colored red by $\sC$.

Consequently, within the graph $\cG$ there is blue clique on $r=L+1$ vertices,
i.e., there is a complete subgraph $\cG^\sharp\subset\cG$ determined by vertices
$\cV^\sharp\defeq\{h_1,\ldots,h_{L+1}\}$ in~$\cV$ such that $\cG^\sharp$
is monochromatic and colored blue by $\sC$.
Ordering $\cV^\sharp$ so that $h_1<\cdots<h_{L+1}$, we have 
\[
\{h_j-h_i:1\le i<j\le m+1\}\subset{\bf Pol}
\]
since the edge determined by $h_i$ and $h_j$ is colored blue by $\sC$
for $1\le i<j\le m+1$. By construction, all of the differences $h_j-h_i$
are distinct from one another. Since $h_i$ and $h_j$ are distinct
vertices in $\cV$, we have $h_{n+1}-h_n\in\cA$ for each $n\in[1,L]$
in view of \eqref{eq:v sight}. Defining $a_n\defeq h_{n+1}-h_n$,
we obtain a sequence $a_1,\ldots,a_L$ with the properties
stated in the theorem, since
\[
\sum_{i\le n\le j}a_n
=\sum_{i\le n\le j}(h_{n+1}-h_n)=h_{j+1}-h_i\in{\bf Pol}.
\]
This completes the proof of Theorem~\ref{thm:hard}.
\end{proof}

\section*{Acknowledgement}

The author thanks Terry Tao for various helpful conversations.

\end{document}